\documentclass{amsart}
\usepackage{amsmath}
\usepackage{amssymb}
\usepackage{amsthm}
\usepackage[dvips]{graphicx}
\def\frk{\mathfrak}               % font for "Fraktur"

\def\frm{{\frk m}}
\def\fraM{{\frk M}}

\def\opn#1#2{\def#1{\operatorname{#2}}} % to make operators
\opn\chara{char} \opn\length{\ell} \opn\pd{pd} \opn\rk{rk}
\opn\projdim{proj\,dim} \opn\injdim{inj\,dim} \opn\rank{rank}
\opn\depth{depth} \opn\codepth{codepth} \opn\grade{grade}
\opn\height{height} \opn\embdim{emb\,dim} \opn\codim{codim}

\opn\Tr{Tr} \opn\bigrank{big\,rank}
\opn\superheight{superheight}\opn\lcm{lcm}
\opn\trdeg{tr\,deg}%
\opn\reg{reg} \opn\lreg{lreg} \opn\skel{skel} \opn\Gr{Gr}
\opn\dim{dim} \opn\indeg{indeg} \opn\Ass{Ass} \opn\Min{Min}
\opn\div{div} \opn\Div{Div} \opn\cl{cl} \opn\Cl{Cl}
\opn\Spec{Spec} \opn\Supp{Supp} \opn\supp{supp} \opn\Sing{Sing}
\opn\Ass{Ass}
\opn\Ann{Ann} \opn\Rad{Rad} \opn\Soc{Soc}
\opn\Sym{Sym} \opn\Ker{Ker} \opn\Coker{Coker} \opn\Im{Im}
\opn\Hom{Hom} \opn\Tor{Tor} \opn\Ext{Ext} \opn\End{End}
\opn\Aut{Aut} \opn\id{id} \opn\ini{in} \opn\tr{tr}

\def\core{core}
\opn\nat{nat}\opn\it{it}
\opn\pff{proof}%   \pf exists already
\opn\Pf{proof} \opn\GL{GL} \opn\SL{SL} \opn\mod{mod} \opn\ord{ord}
\opn\diam{diam}
\opn\dist{dist}
\opn\aff{aff} \opn\con{conv} \opn\relint{relint} \opn\st{st}
\opn\lk{lk} \opn\cn{cn} \opn\core{core} \opn\vol{vol}
\opn\link{link} \opn\star{star} \opn\skel{skel}
\opn\gr{gr}

\def\pot#1#2{#1[\kern-0.28ex[#2]\kern-0.28ex]}

\opn\dirlim{\underrightarrow{\lim}}
\opn\inivlim{\underleftarrow{\lim}}
\def\Implies{\ifmmode\Longrightarrow \else
     \unskip${}\Longrightarrow{}$\ignorespaces\fi}
\def\implies{\ifmmode\Rightarrow \else
     \unskip${}\Rightarrow{}$\ignorespaces\fi}
\def\iff{\ifmmode\Longleftrightarrow \else
     \unskip${}\Longleftrightarrow{}$\ignorespaces\fi}

\let\:=\colon
\theoremstyle{plain}
\newtheorem{thm}{Theorem}[section]
\newtheorem{lemma}[thm]{Lemma}
\newtheorem{prop}[thm]{Proposition}
\newtheorem{cor}[thm]{Corollary}

\newtheorem{conj}[thm]{Conjecture}
\newtheorem{quest}[thm]{Question}

\newtheorem*{thm-q}{Theorem}
\newtheorem*{cor-q}{Corollary}
\newtheorem*{quest-q}{Question}
\newtheorem*{quests-q}{Questions}
\theoremstyle{definition}
\newtheorem{defn}[thm]{Definition}
\newtheorem{exam}[thm]{Example}

\newtheorem*{acknowledgement}{Ackowledgement}
\theoremstyle{remark}
\newtheorem{remark}[thm]{Remark}

\let\epsilon\varepsilon
\let\phi=\varphi
\let\kappa=\varkappa
\def\qed{\ifhmode\textqed\fi
   \ifmmode\ifinner\quad\qedsymbol\else\dispqed\fi\fi}
\def\textqed{\unskip\nobreak\penalty50
    \hskip2em\hbox{}\nobreak\hfil\qedsymbol
    \parfillskip=0pt \finalhyphendemerits=0}
\def\dispqed{\rlap{\qquad\qedsymbol}}

\opn\Gin{Gin}

\opn\inii{in} \opn\inim{inm} \opn\rate{rate}
\numberwithin{equation}{section}

\begin{document}
\title{On the second powers of Stanley-Reisner ideals}
%%%%%%%%%%%%%%%%%%%%%%%%%%%%%%%%%%%%%%%%
%% Information for first author
%%%%%%%%%%%%%%%%%%%%%%%%%%%%%%%%%%%%%%%%
%%%%%%%%%%%%%%%%%%%%%%%%%%%%%%%%%%%%%%%%
%% Information for second author
%%%%%%%%%%%%%%%%%%%%%%%%%%%%%%%%%%%%%%%%
\author[Giancarlo Rinaldo]{Giancarlo Rinaldo}
\address[Giancarlo Rinaldo]{Dipartimento di Matematica, 
Universita' di Messina, Salita Sperone, 31. S. Agata, 
Messina 98166, Italy}
\email{rinaldo@dipmat.unime.it}
%%%%%%%%%%%%%%%%%%%%%%%%%%%%%%%%%%%%%%%%
%% Information for third author
%%%%%%%%%%%%%%%%%%%%%%%%%%%%%%%%%%%%%%%%
\author[Naoki Terai]{Naoki Terai}
\address[Naoki Terai]{Department of Mathematics, Faculty of Culture 
and Education, Saga University, Saga 840--8502, Japan}
\email{terai@cc.saga-u.ac.jp}
%%%%%%%%%%%%%%%%%%%%%%%%%%%%%%%%%%%%%%%%
%% Information for fourth author
%%%%%%%%%%%%%%%%%%%%%%%%%%%%%%%%%%%%%%%%
\author[Ken-ichi Yoshida]{Ken-ichi Yoshida}
\address[Ken-ichi Yoshida]{Graduate School of Mathematics, Nagoya University, 
         Nagoya 464--8602, Japan}
\email{yoshida@math.nagoya-u.ac.jp}
%%%%%%%%%%%%%%%%%%%%%%%%%%%%%%%%%%%%%%%%%
%    General info
%%%%%%% %%%%%%% %%%%%%% %%%%%%% %%%%%%% %%%%%%% %%%%%%% 
 %%% 13D02 Homological methods, Syzygies and resolutions
 %%% 13F55 Face and Stanley-Reisner rings; simplicial complexes 
 %%% 13H10 Special types (Cohen-Macaulay, Gorenstein, Buchsbaum, etc.)
 %%% 05E99 Algebraic combinatorics, None of the above, but in this section
 %%% 13D45 Homological methods, Local cohomology
\date{\today}
%\dedicatory{}

\begin{abstract}
In this paper, we study several properties of 
the second power $I_{\Delta}^2$ of a Stanley-Reisner ideal $I_{\Delta}$ of any dimension.  
As the main result, we prove that $S/I_{\Delta}$ is Gorenstein 
whenever $S/I_{\Delta}^2$ is Cohen-Macaulay over any field $K$.  
Moreover, we give a criterion for the second symbolic power of $I_{\Delta}$ 
to satisfy $(S_2)$
and to coincide with the ordinary power, respectively.  
Finally, we provide new examples of Stanley-Reisner ideals whose second powers 
are Cohen-Macaulay.      
\end{abstract}

\maketitle
\setcounter{section}{-1}
\section{Introduction}

It is proved in \cite{TeTr} that a simplicial complex $\Delta$ is a complete intersection
if the third power $I_{\Delta}^{3}$ of its Stanley-Reisner ideal
is Cohen-Macaulay,
using a result in  \cite{MiT2, Var}.
On the other hand, there is a simplicial complex $\Delta$ which is not a complete intersection
such that $I_{\Delta}^{2}$ is Cohen-Macaulay.
The simplicial complex associated with a pentagon is  such an example. 
Among one-dimensional simplicial complexes, the above example
 is a unique one, as shown in \cite{MiT1}. 
As for the two-dimensional case, such simplicial complexes are  
classified in \cite{TrTu}.
In \cite{MiT2} a characterization of Cohen-Macaulayness of 
the second symbolic power $I_{\Delta}^{(2)}$ is given. 
\par   
A main motivation of this paper is to study 
the Cohen-Macaulayness of the second ordinary powers 
of Stanley-Reisner ideals of any dimension. 
We consider the following two questions:
\begin{enumerate}
\item What constraints does Cohen-Macaulayness of $I_{\Delta}^{2}$ impose upon a simplicial complex $\Delta$?
\item Do there exist \textit{many} simplicial complexes $\Delta$ such that $I_{\Delta}^{2}$ are Cohen-Macaulay?
\end{enumerate}
\par
As for the second question 
we give two families of examples. One is a simplicial join of pentagons;
the other is a stellar subdivision of a complete intersection  complex.

\par
For the first question we treat more general properties and give necessary conditions 
for Cohen-Macaulayness of the square, as a result.
In each section we pick up a different condition; In Sections 2, 3, and 4  we consider  
quasi-Buchsbaum property, Serre's condition $(S_2)$, and unmixedness of a (symbolic) square, respectively.
Summarizing results in these sections, we have the following theorem:

\begin{thm} \label{IntroMain}
Let $\Delta$ be a simplicial complex on $[n]=\{1,2,\ldots,n\}$.
Let $S=K[x_1,\ldots,x_n]$ be a polynomial ring.  
Suppose that $S/I_{\Delta}^2$ is Cohen-Macaulay over any field $K$.   
Then the following conditions are satisfied:
\begin{enumerate}
\item $\Delta$ is Gorenstein.   
\item $\diam ((\link_{\Delta} F)^{(1)}) \le 2$ for any face $F \in \Delta$ with 
$\dim \link_{\Delta} F \ge 1$. 
\item For $F_1,F_2, F_3 \in 2^{[n]} \setminus \Delta $ there exist $G_1, G_2 \in 2^{[n]} \setminus \Delta $ 
such that $G_1 \cup G_2 \subset F_1 \cup F_2 \cup  F_3 $ and $G_1 \cap G_2 \subset F_1 \cap F_2 \cap  F_3 $. 
\end{enumerate}
\end{thm}

As shown in Corollary \ref{S2thm} the condition (2) 
is equivalent to Serre's condition $(S_2)$ of $S/I_{\Delta}^{(2)}$.
And  as shown in Theorem \ref{SpTriangle} the condition (3) 
is equivalent to the condition $I_{\Delta}^{2}=I_{\Delta}^{(2)}$.

We may ask the converse:

\begin{quest}
{\rm
Do the conditions (1), (2) and (3) imply 
that $S/I_{\Delta}^2$ is Cohen-Macaulay?
}
\end{quest}

It is known that
Cohen-Macaulayness of $I_{\Delta}^{2}$ is equivalent to Cohen-Macaulayness of $I_{\Delta}^{(2)}$ and
$I_{\Delta}^{2}=I_{\Delta}^{(2)}$.
Hence the above question will be affirmative if so is  the following one, which is interesting in its own right:

\begin{quest} \label{Q12}
{\rm
Do the conditions (1) and  (2) imply 
that $S/I_{\Delta}^{(2)}$ is Cohen-Macaulay?
}
\end{quest}

Stronger versions of the first question are as follows:

\begin{quest} \label{Q13}
{\rm
Do the conditions (1) and (3) imply 
that $S/I_{\Delta}^2$ is Cohen-Macaulay?
}
\end{quest}

\begin{quest} \label{Q23}
{\rm
Do the conditions (2) and (3) imply 
that $S/I_{\Delta}^2$ is Cohen-Macaulay?
}
\end{quest}
\par
By \cite{MiT1},
the above questions are true if simplicial complexes are one-dimensional.
\par
For the case that edge ideals $I(G)$ of  graphs $G$ without isolated vertices 
are unmixed with the condition $2\height I(G)=n$,
the above questions are also true.
If $I(G)$ is Gorenstein, then it is a complete intersection by \cite{CRT}.
Hence $I(G)^2$ is Cohen-Macaulay and Questions \ref{Q12} and  \ref{Q13} are affirmative.
On the other hand, it is proved in \cite{CRTY} that 
there is some face $F$ in the simplicial complex $\Delta _2$ corresponding to the polarization of the second
symbolic power $I(G)^{(2)}$ such that $\link_{\Delta _2}F$ is not strongly connected,
if $I(G)$ is not a complete intersection.
This implies that the polarization of $I(G)^{(2)}$ does not satisfy  Serre's condition $(S_2)$.
By \cite{MuT}, $I(G)^{(2)}$ does not satisfy  Serre's condition $(S_2)$, either.
It means that $I(G)$ is a complete intersection if $I(G)^{(2)}$ satisfies  Serre's condition $(S_2)$.
Hence Question \ref{Q23} is also affirmative.

\par \vspace{2mm}
Now let us summarize the organization of the paper. 
In Section 1, we fix the terminology which we need later.

\par \vspace{2mm}
In Section 2 we consider quasi-Buchsbaum property, which is weaker than Cohen-Macaulay property.
And we prove the following theorem as a main result in this section:

\par \vspace{2mm} \par \noindent 
{\bf Theorem \ref{Buchsbaum}}
Let $\Delta$ be a simplicial complex on $[n]$ of dimension $d-1 \ge 2$.  
Let $S=K[x_1,\ldots,x_n]$ be a polynomial ring.
Suppose that $S/I_{\Delta}^2$ is quasi-Buchsbaum over any field $K$.   
Then $S/I_{\Delta}$ is Gorenstein.

\par \vspace{2mm}
Since Cohen-Macaulay property implies Serre's condition $(S_2)$, 
in Section 3 we give a criterion for $I_{\Delta}^{(2)}$ to satisfy $(S_2)$, which 
is a generalization of \cite[Theorem 2.3]{MiT2}; see Theorem \ref{Depth} and
Corollary \ref{S2thm}. 
As an application, we show that for Reisner's complex  
(a triangulation of the real projective plane) $\Delta$,
$S/I_{\Delta}^{(2)}$ satisfies $(S_2)$ but is \textit{not}  Cohen-Macaulay. 

\par \vspace{2mm}
In Section 4  we consider the problem 
when $I^{(2)} = I^2$ holds for a Stanley-Reisner ideal $I$, which is also a necessary condition
for Cohen-Macaulayness of $I^{2}$.  
It is also discussed in \cite{TrTu}.
We give a criterion for the second symbolic power 
to be equal to the ordinary power for Stanley-Reisner ideals 
in terms of the hypergraph of the generators; see Theorem \ref{SpTriangle}.  
This generalizes a similar criterion for edge ideals. 
As an application, we show that the second powers of the edge ideals 
of finitely many disjoint union of pentagons are Cohen-Macaulay
as in the second symbolic power case in \cite{MiT2}.

\par \vspace{2mm}
In Section 5, we give examples of the complexes whose second powers of 
the Stanley-Reisner ideals are Cohen-Macaulay. 
More precisely, we prove the following theorem, 
which  is a generalization of a two-dimensional complex  
in \cite[Theorem 3.7 (iii)]{TrTu}.

\par \vspace{2mm} \par \noindent 
{\bf Theorem \ref{Subdiv}.}
Let $\Delta$  be a stellar subdivision of a non-acyclic 
complete intersection complex $\Gamma$. 
Then $S/I_{\Delta}^2$ is Cohen--Macaulay. 

%%%%%%%%%%%%%%%%%%%%%%%%%%%%%%5%%%%%%%%%%%%%%%%%%%%%%%%%%%%%%%%%%%%%%%%%%%%%%%%%%%%
\section{Preliminaries}
In this section we recall several definitions and properties that we
will use later. 
See also \cite{BH, Ma, St, SV}. 

\par \vspace{2mm}
\subsection{Stanley--Reisner ideals}
Let $V=[n]$. 
A nonempty subset $\Delta$ of the power set $2^V$ 
is called a \textit{simplicial complex} on $V$ if 
(i) $F \in \Delta$, $F' \subseteq F \Longrightarrow F' \in \Delta$
and (ii) $\{v\} \in \Delta$ for all $v \in V$.
An element $F \in \Delta$ is called a \textit{face} of $\Delta$.  
The dimension of $F$ is defined by $\dim F = \sharp(F)-1$, where 
$\sharp(F)$ denotes the cardinality of a set $F$.   
The dimension of $\Delta$, denoted by $\dim \Delta$, 
is the maximum of the dimensions of all faces.  
A maximal face of $\Delta$ is called a \textit{facet} of $\Delta$, and  
let $\mathcal{F}(\Delta)$ denote the set of all facets of $\Delta$. 

\par
In the following, let $\Delta$ be a simplicial complex with 
$\dim \Delta =d-1$, and let $K$ be a field.   
Then $\Delta$ is called \textit{pure}  
if all the facets of $\Delta$ have the same cardinality $d$. 
Put $f_i(\Delta)=\sharp\{F \in \Delta\,:\, \dim F =i\}$ 
for each $i=0,1,\ldots,d-1$. 
For each $i$, $\widetilde{H}_i(\Delta; K)$ (resp. $\widetilde{H}^i(\Delta;K)$) 
denotes the $i$th reduced simplicial homology (resp. cohomology) of $\Delta$ 
with values in $K$. 
We omit the symbol $K$ unless otherwise specified.  
The \textit{reduced Euler characteristic} of $\Delta$ is defined by 
\[
\widetilde{\chi}(\Delta) = -1 + \sum_{i=0}^{d-1} f_i(\Delta) = 
\sum_{i=-1}^{d-1} (-1)^i \dim_K \widetilde{H}_i (\Delta).   
\] 
\par
For each face $F \in \Delta$, 
the $\textit{star}$ and the
\textit{link} of $F$ are defined by
\[
\star_{\Delta} F = \{H \in \Delta \;:\,H \cup F \in \Delta\}, \quad
\link_{\Delta} F = \{H \in \star_{\Delta} F \;:\, H \cap F = \emptyset\}.
\]
Note that these are also simplicial complexes. 
Moreover, we note that for any subset $W \subseteq V$, 
$\Delta_W = \{F \in \Delta \,:\, F \subseteq W \}$ 
is also a subcomplex of $\Delta$. 
For any integer $k$ with $0 \le k \le d-1$, the $k$-th \textit{skeleton} of 
$\Delta$ is defined by 
$\Delta^{(k)} = \{F \in \Delta \,;\, \dim F \le k \}$.
Then $\Delta^{(k)}$ is a subcomplex of $\Delta$ with $\dim \Delta^{(k)} = k$. 
\par 
The \textit{Stanley--Reisner ideal} of $\Delta$, denoted by $I_{\Delta}$,   
is the squarefree monomial ideal of $S=K[x_1,\ldots,x_n]$ generated by 
\[
 \{x_{i_1} x_{i_2} \cdots x_{i_p} \,:\, 1 \le i_1 < \cdots < i_p \le n,\; 
\{x_{i_1},\ldots,x_{i_p}\} \notin \Delta \},  
\]
and $K[\Delta]= K[x_1,\ldots,x_n]/I_{\Delta}$ is called 
the \textit{Stanley--Reisner ring} of $\Delta$. 
Note that the Krull dimension of $K[\Delta]$ is equal to $d$. 
For any subset $\sigma$ of $V$, $x_{\sigma}$ denotes 
the squarefree monomial in $K[x_1,\ldots,x_n]$ with support $\sigma$.  
\par
For a simplicial complex $\Delta$ on $V$, we put  
$\core V = \{x \in V\,:\, \star\{x\} \ne V \}$. 
Moreover, we define the \textit{core} of $\Delta$ by 
$\core \Delta = \Delta_{\core V}$. 
\par
For a given face $F$ of $\Delta$ with $\dim F \ge 1$ and 
a new vertex $v$, the \textit{stellar subdivision} of $\Delta$ 
on $F$ is the simplicial complex $\Delta_{F}$ on the vertex 
set $V \cup \{v\}$ defined by 
\[
\Delta_{F} = \big(\Delta \setminus 
\{H \,|\, F \subseteq H \in \Delta\} \big)
\cup 
\{H \cup \{v\}\,|\, 
H \in \Delta,\,F \not \subseteq H,\, 
F \cup H \in \Delta \}.
\]
Notice that $\Delta_{F}$ is homeomorphic to $\Delta$.

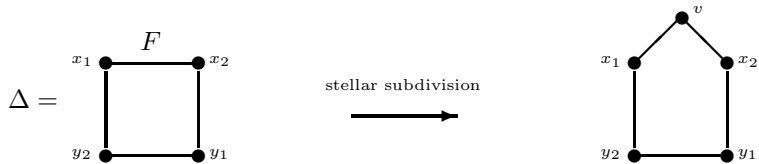
\par \vspace{2mm}  
\begin{picture}(400,70)
%%%%%%%%%%%%%%%%%%%%%%%%%%%%%%%%%%%%%%%%
  \put(40,28){$\Delta=$}
  \thicklines
  \put(77,45){\circle*{5}}  %1
  \put(77,10){\circle*{5}}  %2  
  \put(112,45){\circle*{5}}  %3
  \put(112,10){\circle*{5}}  %4  
  \put(64,45){{\tiny $x_1$}}  %1
  \put(64,10){{\tiny $y_2$}}  %2
  \put(116,45){{\tiny $x_2$}}  %1
  \put(116,10){{\tiny $y_1$}} 
  \put(77,42){\line(0,-1){29}}  
  \put(112,42){\line(0,-1){29}} 
  \put(77,45){\line(1,0){35}} 
  \put(77,10){\line(1,0){35}} 
  \put(90,50){$F$}
  \put(170,25){\vector(1,0){40}}
  \put(160,35){{\tiny stellar subdivision}}
  \put(277,45){\circle*{5}}  %1
  \put(277,10){\circle*{5}}  %2  
  \put(312,45){\circle*{5}}  %3
  \put(312,10){\circle*{5}}  %4  
  \put(295,62){\circle*{5}}
  \put(264,45){{\tiny $x_1$}}  %1
  \put(264,10){{\tiny $y_2$}}  %2
  \put(316,45){{\tiny $x_2$}}  %1
  \put(316,10){{\tiny $y_1$}} 
  \put(299,64){{\tiny $v$}}
  \put(277,42){\line(0,-1){29}}  
  \put(312,42){\line(0,-1){29}} 
  \put(277,10){\line(1,0){35}} 
  \put(277,45){\line(1,1){18}} 
  \put(312,45){\line(-1,1){18}}
%%%%%%%%%%%%%%%%%%%%%%%%%%%%%%%%%%%%%%%%
\end{picture}

%%%%%%%%%%%%%%%%%%%%%%%%%%%%%%%%%%%%%%%%%%
\par \vspace{2mm}
Let $G$ be a graph, which means a  finite graph without loops and 
multiple edges.  
Let $V(G)$ (resp. $E(G)$) denote the set of vertices (resp. edges) of $G$. 
Put $V(G) =[n]$. 
Then the \textit{edge ideal} of $G$, denoted by $I(G)$, 
is a squarefree monomial ideal 
of $S=K[x_1,\ldots,x_n]$ defined by 
\[
 I(G) = (x_ix_j \,:\, \{i, j\} \in E(G)).
\]

\par 
For an arbitrary graph $G$, the simplicial complex $\Delta(G)$ 
with $I(G) = I_{\Delta(G)}$ is called 
the \textit{complementary simplicial complex} of $G$. 

\par \vspace{2mm}
Let $G$ be a connected graph, and let $p$,$q$ be two vertices of $G$.  
The \textit{distance} between $p$ and $q$, denoted by $\dist(p,q)$, 
is the minimal length of paths from $p$ to $q$. 
The \textit{diameter}, denoted by $\diam G$, is the maximal distance 
between two vertices of $G$. 
We set $\diam G = \infty$ if $G$ is a disconnected graph. 
\par \vspace{2mm}
Let $\Delta$ be a simplicial complex on $V$ of dimension $1$. 
Then $\Delta$ can be regarded as a graph on $V$ whose 
edge set is defined by $E(\Delta) = \{F \in \Delta \,:\, \dim F=1 \}$.

%%%%%%%%%%%%%%%%%%%%%%%%%%%%%%%%%%%%%%%%%%%%%%
\par \vspace{2mm}
\subsection{Symbolic powers}
Let $I$ be a radical ideal of $S$. 
Let $\Min_S(S/I) = \{P_1,\ldots, P_r\}$ be the set of the 
minimal prime ideals of $I$, and 
put $W = S \setminus \bigcup_{i=1}^r P_i$. 
Given an integer $\ell \ge 1$, 
the \textit{$\ell$th symbolic power} of $I$ 
is defined to be the ideal
\[
 I^{(\ell)}= I^{\ell}S_W \cap S = \bigcap_{i=1}^r P_i^{\ell}S_{P_i} \cap S.
\]
In particular, if $I=I_{\Delta}$ is the Stanley-Reisner ideal of $\Delta$, 
putting $P_F=(x \in [n] \setminus F)$ for each facet $F$, then we have 
\[
 I_{\Delta} = \bigcap_{F \in \mathcal{F}(\Delta)} P_F
\]
and hence 
\[
 I_{\Delta}^{(\ell)} = \bigcap_{F \in \mathcal{F}(\Delta)} P_F^{\ell}.   
\]
\par
In general, $I^{\ell} \subseteq I^{(\ell)}$ holds, but the other inclusion 
does not necessarily hold. 
For instance, if $I=(x_1x_2,x_2x_3,x_3x_1)$, then 
\[
 I^{(2)} = (x_1,x_2)^2 \cap (x_2,x_3)^2 \cap (x_1,x_3)^2 
= I^2 + (x_1x_2x_3) \ne I^2. 
\]
\par
Moreover, if $I$ is a unmixed squarefree monomial ideal, 
then $I^{(\ell)}$ is unmixed. 
Thus if $S/I^{\ell}$ is Cohen-Macaulay (or Buchsbaum), 
then so is $S/I^{(\ell)}$. 

%%%%%%%%%%%%%%%%%%%%%%%%%%%%%%%%%%%%%%%%%%
\par \vspace{2mm}
\subsection{Serre's condition}
Let $S=K[x_1,\ldots,x_n]$ and $\frm = (x_1,\ldots,x_n)S$. 
Let $I$ be a homogeneous ideal of $S$. 
For a positive integer $k$,   
$S/I$ satisfies \textit{Serre's condition} $(S_k)$ 
if $\depth (S/I)_P \ge \min\{\dim (S/I)_P,\, k \}$ for every $P \in \Spec S/I$.

\par
A simplicial complex $\Delta$ is called 
\textit{Cohen--Macaulay} (resp. Gorenstein, (FLC) etc.) 
if so is $K[\Delta]$ over any field $K$.  
Moreover, 
if $\Delta$ is (FLC), then $\Delta$ is pure and 
$\link_{\Delta}(F)$ is Cohen-Macaulay for every nonempty face $F \in \Delta$. 
\par 
A homogeneous $K$-algebra $S/I$ is called \textit{quasi-Buchsbaum}
if $\frm H_{\frm}^i(S/I) =0$ for each $i=0,1,\ldots,\dim S/I-1$. 
It is known that any quasi-Buchsbaum ring has (FLC) and the converse 
is also true for Stanley-Reisner rings.

%%%%%%%%%%%%%%%%%%%%%%%%
\par \vspace{2mm}
\subsection{Associated simplicial complex of monomial ideals}

Let $S=K[x_1,\ldots,x_n]$ be a polynomial ring 
with natural $\mathbb{Z}^n$-graded structure. 
Let $\frm=(x_1,\ldots,x_n)S$ be the unique homogeneous maximal 
ideal of $S$.  
Let $I$ be a monomial ideal of $S$, and 
let $G(I)$ denote the minimal monomial generators of $I$.  
For each $i$, we put $\rho_i= \max\{b_i \,:\, x^{\bf b} \in G(I)\}$,
where ${\bf b} = (b_1,\ldots,b_n) \in \mathbb{N}^n$ and $x^{\bf b} = x_1^{b_1}\cdots x_n^{b_n}$. 
Then $S/I$ can be considered as a $\mathbb{Z}^n$-graded ring. 
\par
Let ${\bf a} \in \mathbb{Z}^n$ be a vector. 
For any $\mathbb{Z}^n$-graded $S$-module $M$,  
$M_{\bf a}$ denotes the graded ${\bf a}$-component of $M$. 
We put $G_{\bf a} = \{i \in [n]\,:\, a_i < 0 \}$. 
As $\sqrt{I}$ is a squarefree monomial ideal, 
there exists a simplicial complex $\Delta$ 
such that $I_{\Delta} = \sqrt{I}$. Then we define 
$\Delta(I) = \Delta$. 
Under this notation, a subcomplex $\Delta_{\bf a}(I)$ is defined by 
\[
\Delta_{\bf a}(I) = \left\{F \in \Delta(I) \,:
\begin{array}{l}
\bullet \;\text{$F \cap G_{\bf a} = \emptyset$.}\\
\bullet \;\text{For every $x^{\bf b} \in G(I)$, 
there exists an $i \in [n] \setminus (F \cup G_{\bf a})$} \\
\phantom{\bullet}\; \text{such that $b_i > a_i$}. 
\end{array}
\right\}.
\]

\par
This complex plays a key role in Takayama's formula for local cohomology modules 
of monomial ideals, which is known 
as Hochster's formula in the case of squarefree monomial ideals.

\par \vspace{2mm}
Let $I=I_{\Delta}$ be a squarefree monomial ideal of $S$. 
Then $I^{(\ell)}$ is a monomial ideal whose radical is equal to $I$. 
The following lemma enables us to compute $\Delta_{\bf a}(I^{(\ell)})$ easily. 

\begin{lemma}[Minh and Trung \cite{MiT1}] \label{MTrep}
Let $I$ be a squarefree monomial ideal in $S$. 
Let $\ell \ge 1$ be an integer and ${\bf a} \in \mathbb{N}^n$. 
Then we have 
\[
 \Delta_{\bf a}(I^{(\ell)})
 = \langle F \in \mathcal{F}(I) \,:\, 
\sum_{i \notin F} a_i \le \ell-1 \rangle. 
\]
\end{lemma}

%%%%%%%%%%%%%%%%%%%%%%%%%%%%%%%%%%%%
\subsection{Linkage}

Let $R$ be a Gorenstein ring, and $I$, $J$ ideals of $R$. 
$I$ and $J$ said to be \textit{directly linked}, denoted by $I \sim J$, 
if there exists a regular sequence 
$\underline{z}=z_1,\ldots,z_h$ in $I \cap J$ such that 
$J = (\underline{z}) \colon I$ and $I = (\underline{z}) \colon J$. 
\par
Assume that $I$ is Cohen-Macaulay ideal of height $h$ and 
$\underline{z}=z_1,\ldots,z_h$ is a regular sequence 
contained in $I$. 
If we set $J = (\underline{z}) \colon I$, then 
$I=(\underline{z}) \colon J$ and thus $I \sim J$.  
\par 
Moreover, $I$ is said to be \textit{linked} to $J$ 
(or $I$ lies in the linkage class of $J$) if 
there exists a sequence of ideals of direct links 
\[
 I = I_0 \sim I_1 \sim \cdots \sim I_r =J. 
\] 
One can easily see that $\sim$ is an equivalence relation of ideals and 
any two complete intersection with the same height belongs to the same class.  
In particular, $I$ is called \textit{licci} if 
$I$ lies in the linkage class of a complete intersection ideal. 
See e.g. \cite{Vas} for more details.

%%%%%%%%%%%%%%%%%%%%%%%%%%%%%%%%%%%%%%%%%%%%%%%%%%%%%%%%%%%%%%%%%%%%%%%%%%%%%%%
%%%%%%%%%%%%%%%%%%%%%%%%%%%%%%%%%%%%%%%%%%%%%%%%%%%%%%%%%%%%%%%%%%%%%%%%%%%%%%%
\medskip
\section{Quasi-Buchsbaumness of the second powers and Gorensteinness}

In this section we consider quasi-Buchsbaum property of the second power of 
the Stanley-Reisner ideal $I_{\Delta}$. 
The main purpose of this section is to prove the 
following theorem$:$
 
\begin{thm} \label{Buchsbaum}
Let $S=K[x_1,\ldots,x_n]$ be a polynomial ring over a field $K$, 
and let $\Delta$ be a simplicial complex on $V=[n]$.  
Suppose that $d = \dim S/I_{\Delta} \ge 3$. 
If $S/I_{\Delta}^2$ is quasi-Buchsbaum for any field $K$
then $\Delta$ is Gorenstein.  
\end{thm}

%%%%%%%%%%%%%%%%%%%%%%%%%%%%%%%%%%%%%%%%%%%%%%%%%%%%%%%%%%%
\par \vspace{2mm}
We first prove the following lemma, 
which is closely related to the conjecture 
by Vasconcelos (see also \cite[Conjecture 3.12]{SVV}):
Let $R$ be a regular local ring and $I$ a Cohen-Macaulay 
ideal of $R$. If $I$ is syzygetic and $I/I^2$ is Cohen-Macaulay, 
then $I$ is a Gorenstein ideal.  
The following lemma easily follows from the classification theorems for
simplicial complexes $\Delta$ such that $S/I_{\Delta}^2$ are Cohen-Macaulay
in one and  two-dimensional cases.
See \cite{MiT1, TrTu}.

\begin{lemma} \label{Vasconcelos}
Let $\Delta$ be a simplicial complex on $V=[n]$, and 
let $I_{\Delta} \subseteq S=K[x_1,\ldots,x_n]$ denote 
the Stanley-Reisner ideal of $\Delta$. 
If $S/I_{\Delta}^2$ is Cohen-Macaulay for any field $K$, then 
$\Delta$ is Gorenstein. 
\end{lemma}

\begin{proof}
We may assume that $\Delta = \core \Delta$. 
Let $K$ be a field and fix it. 
Let $F$ be a face of $\Delta$ and 
put  $\Gamma = \link_{\Delta} F$.   
\par \vspace{2mm}
First note that $S/I_{\Gamma}^2$ and $S/I_{\Delta}$ are Cohen-Macaulay 
if so is $S/I_{\Delta}^2$. 
Indeed, since $S/I_{\Delta}^2$ is Cohen-Macaulay and 
$I_{\Delta} = \sqrt{I_{\Delta}^2}$, we 
have that $S/I_{\Delta}$ is Cohen-Macaulay; see e.g. \cite{HTT}. 
On the other hand, by localizing at $x_F=\prod_{i \in F} x_i$, we get 
\[
  I_{\Delta} S[x_F^{-1}] 
= (I_{\Gamma}, x_{i_1},\ldots,x_{i_k})S[x_F^{-1}]
\]
for some variables $x_{i_1},\ldots,x_{i_k}$. 
Hence the assumption implies that 
$(I_{\Gamma}, x_{i_1},\ldots,x_{i_k})^2$ is a Cohen-Macaulay ideal. 
This yields that $I_{\Gamma}^2$ is also Cohen-Macaulay.  
\par \vspace{2mm}
Suppose that $\dim \Gamma =0$. 
Then one can take a complete graph $G$ such that $I(G)=I_{\Gamma}$. 
Since $S/I(G)^2$ is Cohen-Macaulay, we have $I(G)^{(2)} = I(G)^2$.  
Hence $G$ does not contain any triangle (e.g. see Corollary \ref{Triangle}).  
Thus $\sharp(V(\Gamma)) = \sharp(V(G)) \le 2$. 
\par \vspace{2mm}
By the above argument, $\Lambda=\link_{\Delta} F$ 
is a locally complete intersection complex whenever $\dim \Lambda =1$. 
Moreover, since $S/I_{\Lambda}$ is Cohen-Macaulay 
and thus $\Lambda$ is connected, 
$\Lambda$ is an $n$-cycle or an $n$-pointed path; 
see \cite[Proposition 1.11]{TY}.
On the other hand, since $\diam \Lambda \le 2$ by \cite[Theorem 2.3]{MiT1}, 
we get
$n \le 3$ if $\Lambda$ is an $n$-pointed path.  
Hence $\Lambda = \link_{\Delta} F$ is Gorenstein. 
\par \vspace{2mm} 
Now suppose that $K=\mathbb{Z}/2\mathbb{Z}$. 
By \cite[Chapter II, Theorem 5.1]{St}, $K[\Delta]$ is Gorenstein. 
Then we get $\widetilde{\chi}(\Delta)=(-1)^{d-1}$.
\par 
Let $K$ be any field. 
Then $\widetilde{\chi}(\Delta)=(-1)^{d-1}$ because 
$\widetilde{\chi}(\Delta)$ does not depend on $K$.  
Therefore we conclude that $\Delta$ is Gorenstein over $K$ 
by \cite[Chapter II, Theorem 5.1]{St} again. 
\end{proof}

\par 
A complex $\Delta$ is called a \textit{locally Gorenstein} complex  
if $\link_{\Delta}\{x\}$ is Gorenstein for every vertex $x \in V$. 
Then the following corollary immediately follows from 
Lemma \ref{Vasconcelos}.

\begin{cor} \label{FLC-Vas}
If $S/I_{\Delta}^2$ has $($FLC$)$ for any field $K$, 
then $\Delta$ is a locally Gorenstein complex. 
\end{cor}

\begin{proof} 
The assumption implies that $S/I_{\link_{\Delta}\{x\}}^2$ is Cohen-Macaulay
for every vertex $x \in V$. 
Then $\link_{\Delta}\{x\}$ is Gorenstein by Lemma \ref{Vasconcelos}.  
\end{proof}

\begin{lemma} \label{Q-Bbm}
Suppose $d \ge 2$. 
If $S/I_{\Delta}^2$ is quasi-Buchsbaum, then $S/I_{\Delta}$ is Cohen-Macaulay.  
\end{lemma}

\begin{proof}
By assumption that $S/I_{\Delta}^2$ has (FLC).    
Then $S/I_{\Delta}$ has (FLC) by \cite[Theorem 2.6]{HTT} 
and thus it is Buchsbaum. 
\par 
Now suppose that $S/I_{\Delta}$ is \textit{not} Cohen-Macaulay. 
Then there exists an $i$ with $0 \le i \le d-2$ such that 
$H_{\frm}^{i+1}(S/I_{\Delta})_0 \cong \widetilde{H}_i(\Delta;K) \ne 0$. 
Then we get the following commutative diagram (see \cite{MiN})

\begin{picture}(400,70)
\put(120,50){$H_{\frm}^{i+1}(S/I_{\Delta}^{2})_{\bf 0}$} 
\put(185,55){\vector(1,0){40}} 
\put(145,42){\vector(0,-1){20}}
\put(200,58){$x_1$}
\put(240,50){$H_{\frm}^{i+1}(S/I_{\Delta}^{2})_{{\bf e}_1}$} 
\put(116,13){$\widetilde{H}^i(\Delta_{\bf 0}(I_{\Delta}^2))$} 
\put(190,15){\vector(1,0){35}} 
\put(270,42){\vector(0,-1){20}}
\put(243,13){$\widetilde{H}^i(\Delta_{{\bf e}_1}(I_{\Delta}^2))$,} 
\end{picture}
\par \vspace{2mm} \par \noindent 
where the bottom map is identity because 
$\Delta_{\bf 0}(I^{2}) = \Delta_{{\bf e}_1}(I^2) = \Delta$ by \cite{TeTr} and 
the vertical maps are isomorphism.    
This yields $x_1 H_{\frm}^{i+1}(S/I_{\Delta}^2) \ne 0$. 
But this contradicts the assumption. 
\end{proof}

\begin{remark}
{\rm 
We have an analogous result in the symbolic power case. 
Namely,
if $S/I_{\Delta}^{(2)}$ is quasi-Buchsbaum, then $S/I_{\Delta}$ is Cohen-Macaulay.
The proof is almost the same since we have  
$\Delta_{\bf 0}(I^{(2)}) = \Delta_{{\bf e}_1}(I^{(2)}) = \Delta$. 
} 
\end{remark}

\par 
We are now ready to prove Theorem \ref{Buchsbaum}.

\begin{proof}[Proof of Theorem \ref{Buchsbaum}]
By assumption and Corollary  \ref{FLC-Vas}, we have that 
$\Delta$ is locally Gorenstein. 
Moreover, $\Delta$ is Cohen-Macaulay by Lemma \ref{Q-Bbm}. 
Take any face $F$ of $\Delta$ with $\dim \link_{\Delta} F =1$. 
As $d \ge 3$, $\link_{\Delta} F$ is given by some 
link of $\link_{\Delta} \{x\}$ for $x \in F$. 
Hence such a $\link_{\Delta} F$ is also Gorenstein. 
By a similar argument as in the proof of Lemma \ref{Vasconcelos}, 
we get the required assertion. 
\end{proof}

\par \vspace{2mm}
The Gorensteinness of $S/I_{\Delta}$ does not necessarily imply 
the quasi-Buchsbaumness of $S/I_{\Delta}^2$.

\par 
We cannot replace the Cohen-Macaulayness of $S/I_{\Delta}^2$ with that of 
$S/I_{\Delta}^{(2)}$ in Lemma \ref{Vasconcelos} as the next example shows.  

\begin{exam} \label{phantomPentagon}
Let $k \ge 2$ be a given integer. 
Let $I$ be the Stanley-Reisner ideal of the following simplicial complex $\Delta$,
Then since $\diam \Delta \le 2$, $S/I^{(2)}$ is Cohen-Macaulay by \cite{MiT1}, 
but $S/I^2$ is not.
Moreover, $S/I$ is not Gorenstein. 

\par 
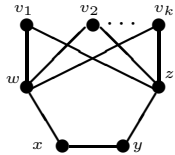
\begin{picture}(400,70)
%%%%%%%%%%%%%%%%%%%%%%%%%%%%
%  \put(265,60){$$}
  \thicklines
  \put(195,55){\circle*{5}}  %1
  \put(170,31.5){\circle*{5}}  %2
  \put(183.5,9){\circle*{5}}  %3
  \put(206.5,9){\circle*{5}}  %4
  \put(220,31.5){\circle*{5}}  %5
  \put(170,55){\circle*{5}}
  \put(220,55){\circle*{5}}
  \put(200,53){$\cdots$}
  \put(165,60){{\tiny $v_1$}}  %v_1
  \put(190,60){{\tiny $v_2$}}  %v_2
  \put(218,60){{\tiny $v_k$}}  %v_k
  \put(162,33){{\tiny $w$}}  %2
  \put(172,8){{\tiny $x$}}  %3
  \put(210,8){{\tiny $y$}}  %4
  \put(222,35){{\tiny $z$}}  %5
%  \put(192,44){\line(-2,-1){20.5}}  %{1, 2}
  \put(170,32){\line(1,1){22}} 
  \put(170,31.5){\line(0,1){22}} 
%  \put(198,44){\line(2,-1){20.5}}  %{1, 5}
  \put(220,32){\line(-1,1){22}} 
  \put(220,31.5){\line(0,1){22}}
  \put(182,11){\line(-3,5){10.8}}  %{2, 3}
  \put(208,11){\line(3,5){10.8}}  %{4, 5}
  \put(186,9){\line(1,0){18}}  %{3, 4}
  \put(170,55){\line(2,-1){50}}
  \put(220,55){\line(-2,-1){50}}
%%%%%%%%%%%%%%%%%%%%%%%%%%%%%%%%
\end{picture}
\end{exam}

\par \vspace{2mm}
In Theorem \ref{Buchsbaum}, we cannot remove the assumption that 
$\dim S/I_{\Delta} \ge 3$ as the next example shows. 

\begin{exam} \label{Counter-Bbm}
Put $I_{\Delta} = (x_1x_3,x_1x_4,x_2x_4)$, the Stanley-Reisner ideal 
of the $4$-pointed path $\Delta$. 
Then $S/I_{\Delta}^2$ is Buchsbaum by \cite[Example 2.9]{TY} and 
$S/I_{\Delta}$ is Cohen-Macaulay but not Gorenstein of dimension $2$. 

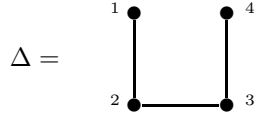
\begin{picture}(400,60)
  \thicklines
  \put(80,25){$\Delta=$}
  \put(127,45){\circle*{5}}  %1
  \put(162,45){\circle*{5}}  %4
  \put(127,10){\circle*{5}}  %2
  \put(162,10){\circle*{5}}  %3  
  \put(118,45){{\tiny $1$}}  %1
  \put(118,10){{\tiny $2$}}  %2
  \put(169,45){{\tiny $4$}}  %4
  \put(169,10){{\tiny $3$}}  %3
%  \put(130,45){\line(1,0){29}}  %{1, 4}
  \put(130,10){\line(1,0){29}}  %{2, 3}
  \put(127,42){\line(0,-1){29}}  %{1, 2}
  \put(162,42){\line(0,-1){29}}  %{4, 3}
  %\put(128,60){{\tiny $3-$path}}
\end{picture}
\end{exam}

\vspace{2mm}
The following question is valid in the case that char $K=2$,
but the other cases remain open. 

\begin{quest} \label{G-quest}
If $S/I_{\Delta}^2$ is Cohen-Macaulay over a fixed field $K$, 
then is $\Delta$ Gorenstein over $K$?
\end{quest}

%%%%%%%%%%%%%%%%%%%%%%%%%%%%%%%%%%%%%%%%%%%%%%%%%%%%%%%%%%%%%%%%%%%%%%%%%%%%%
%%%%%%%%%%%%%%%%%%%%%%%%%%%%%%%%%%%%%%%%%%%%%%%%%%%%%%%%%%%%%%%%%%%%%%%%%%%%%
\medskip
\section{Cohen-Macaulayness versus $(S_2)$ for second symbolic powers}

\par 
Throughout this section, let $S=K[x_1,\ldots,x_n]$ be a polynomial ring 
over a field $K$. 
Let $\frm=(x_1,\ldots,x_n)S$ be the unique graded maximal ideal of $S$ 
with natural graded structure.  
\par 
In \cite{TeTr} it is proved that for any integer $\ell \ge 3$ and 
for any simplicial complex $\Delta $ on the vertex set $V=[n]$,
$S/I_{\Delta }^{(\ell)}$ is Cohen-Macaulay if and only if it satisfies Serre's  condition $(S_2)$. 
So it is natural to ask the following question. 

\begin{quest} \label{S2-quest}
Let $I$ be the Stanley-Reisner ideal of a simplicial complex $\Delta$ on $V=[n]$.    
Then $S/I^{(2)}$ is Cohen-Macaulay if and only if $S/I^{(2)}$ satisfies $(S_2)$?  
\end{quest}

\par \vspace{2mm}
So the aim of this section is to give a criterion 
for $S/I_{\Delta}^{(2)}$ to satisfy $(S_2)$. 
In order to do that, we prove the following theorem, which is a generalization 
of \cite[Theorem 2.3]{MiT1}. 
Using this, we give a negative answer to the above question; see Example \ref{RealP}. 
Note that in the following Theorem \ref{Depth} and Corollary \ref{S2thm}
if we replace the condition that the diameter is less than or equal to $2$ 
by the connectedness condition
then we have the corresponding condition for  the original Stanley-Reisner ring 
instead of the second symbolic power,
e.g., $\depth S/I_{\Delta} \ge 2$ is equivalent to the connectedness of $\Delta $
if $\dim \Delta \ge 1$. 

\begin{thm} \label{Depth}
Let $\Delta$ be a simplicial complex with $\dim \Delta \ge 1$. 
Then the following conditions are equivalent$:$
\begin{enumerate}
 \item $\depth S/I_{\Delta}^{(2)} \ge 2$ 
$($equivalently, $\depth (S/I_{\Delta}^{(2)})_{\frm} \ge 2$$)$. 
 \item $\diam \Delta^{(1)} \le 2$, where $\Delta^{(1)}$ denotes the $1$-skeleton of $\Delta$. 
\end{enumerate}
\end{thm}

\begin{proof} 
Put $\Delta_{\bf a} := 
\langle F \in \mathcal{F}(\Delta) \,:\, \sum_{i \notin F} a_i \le 1 \rangle$.
\par \vspace{2mm} \par \noindent 
(1) $\Longrightarrow (2):$ 
For given $r,\,s \in V=[n]$ ($r < s)$, we show that $\dist(r,s) \le 2$ 
in $\Delta^{(1)}$. 
Put ${\bf a} = {\bf e}_r +  {\bf e}_s \in \mathbb{N}^n$.  
Then 
$\Delta_{\bf a}
=\langle F \in \mathcal{F}(\Delta) \,:\, 
r \in F \;\text{or} \; s \in F \rangle$.  
Since $\depth S/I_{\Delta}^{(2)} \ge 2$, 
we have that $\widetilde{H}_0(\Delta_{\bf a}) =0$ and thus 
$\Delta_{\bf a}$ is connected by Takayama's formula and 
Lemma \ref{MTrep}.
Hence there exists an $F \in \mathcal{F}(\Delta)$ such that $r,s \in F$ 
or there exist $F_r \in \mathcal{F}(\Delta)$ and 
$F_s \in \mathcal{F}(\Delta)$ such that $r \in F_r$, $s \in F_s$ and 
$F_r \cap F_s \ne \emptyset$.  
In any case, we get $\dist(r,s) \le 2$, as required. 
\par \vspace{2mm}
$(2) \Longrightarrow (1):$
Assume $\diam \Delta^{(1)} \le 2$. 
By Takayama's formula, it suffices to show that $\Delta_{\bf a}$ 
is connected for any ${\bf a} \in \{0,1\}^n$ 
with $\Delta_{\bf a} \ne \emptyset$; see also \cite{MiT2}.  
\begin{description}
\item[{\bf Case 1}] $\sharp(\supp {\bf a}) \le 1$.    
\end{description}
Then $\Delta_{\bf a}=\Delta$ is connected by assumption. 
\begin{description}
\item[{\bf Case 2}] $\sharp(\supp {\bf a}) = 2$.    
\end{description}
We may assume that $a_r = a_s=1$ for some $r < s$. 
Then 
\[
\Delta_{\bf a} = \langle F \in \mathcal{F}(\Delta) 
\,:\, r \in F \;\text{or} \;s \in F \rangle.
\] 
Since $\diam \Delta^{(1)} \le 2$, we have that $\{r,s\} \in \Delta$ or 
there exists a $t \in V$ such that $\{r,t\}$, $\{t,s\} \in \Delta$. 
In the first case, if we choose a facet $F \in \mathcal{F}(\Delta)$ which 
contains $\{r,s\}$, then $F \in \Delta_{\bf a}$ and $r,s \in F$. 
In the second case, if we choose facets $F_1$, $F_2$ such that 
$\{r,t\} \in F_1$ and $\{s,t\} \in F_2$. 
Then $\Delta_{\bf a}$ is connected 
because $F_1,\,F_2 \in \Delta_{\bf a}$. 
\begin{description}
\item[{\bf Case 3}] $\sharp(\supp {\bf a}) \ge 3$.    
\end{description}
We may assume that $\sharp(\mathcal{F}(\Delta_{\bf a})) \ge 2$. 
Let $F_1$, $F_2 \in \mathcal{F}(\Delta_{\bf a})$. 
By assumption, $\sharp(F_i \cap \supp({\bf a})) \ge \sharp(\supp({\bf a}))-1$
for each $i=1,2$. 
Then we get 
\[
\sharp(F_1 \cap F_2) \ge \sharp
\big(F_1 \cap \supp({\bf a})) \cap (F_2 \cap \supp({\bf a}))\big) 
\ge \sharp(\supp({\bf a})) -2 \ge 1. 
\]
Hence $\Delta_{\bf a}$ is connected.   
\end{proof}

\begin{cor}  \label{S2thm}
Let $\Delta$ be a pure simplicial complex. 
Then the following conditions are equivalent$:$
\begin{enumerate}
\item $S/I_{\Delta}^{(2)}$ satisfies $(S_2)$.
\item $\diam ((\link_{\Delta} F)^{(1)}) \le 2$ for any face $F \in \Delta$ with 
$\dim \link_{\Delta} F \ge 1$. 
\end{enumerate}
\end{cor}

\begin{proof}
$(1) \Longrightarrow (2):$ 
Let $F$ be a face of $\Delta$ with $\dim \link_{\Delta} F \ge 1$. 
By assumption and localization, 
we obtain that $S'/I_{\link_{\Delta}(F)}^{(2)}$ satisfies $(S_2)$, 
where $S'$ is a polynomial ring which corresponds to $\Gamma = \link_{\Delta}(F)$.  
Then $\depth S'/I_\Gamma^{(2)} \ge 2$. 
It follows from Theorem \ref{Depth} that $\diam \Gamma^{(1)} \le 2$, as required. 
\par \vspace{2mm}
$(2) \Longrightarrow (1):$ 
The assumption (2) preserves under localization. 
Hence we may assume that $S/I_{\link_{\Delta} \{x\}}^{(2)}$ satisfies $(S_2)$.
This implies that $S/I_{\link_{\Delta} \{x\}}$ also satisfies $(S_2)$ by \cite{HTT}. 
Hence $(S/I_{\Delta}^{(2)})_x$ satisfies $(S_2)$ for every variable $x$. 
\par 
Let $P \in \Spec(S/I_{\Delta}^{(2)})$ with $\dim (S/I_{\Delta}^{(2)})_P \ge 2$. 
If $P \ne \frm$, then there exists a variable $x$ such that $x \notin P$. 
Then $\depth (S/I_{\Delta}^{(2)})_P \ge 2$ by the above argument. 
Otherwise, $P = \frm$. 
Since $\diam \Delta^{(1)} \le 2$ by assumption, we have that 
$\depth (S/I_{\Delta}^{(2)})_{\frm} \ge 2$ by Theorem \ref{Depth}. 
Therefore $S/I_{\Delta}^{(2)}$ satisfies $(S_2)$. 
\end{proof}

\par \vspace{2mm}
The next example shows that 
the $(S_2)$-ness of $I_{\Delta}^{(2)}$ does not necessarily imply its Cohen-Macaulayness.

\begin{exam}[{\bf The triangulation of the real projective plane}] \label{RealP}
Let $I=I_{\Delta}$ be the Stanley-Reisner ideal of the triangulation of 
the real projective plane $\mathbb{P}^2$.  
Then $I_{\Delta}$ is generated by the following monomials of degree $3$:  
\[
x_1x_2x_3,\,x_1x_2x_5,\,x_1x_3x_6,\,
x_1x_4x_5,\,x_1x_4x_6,\,x_2x_3x_4,\,
x_2x_4x_6,\,x_2x_5x_6,\,x_3x_4x_5,\,x_3x_5x_6.
\]
\par 
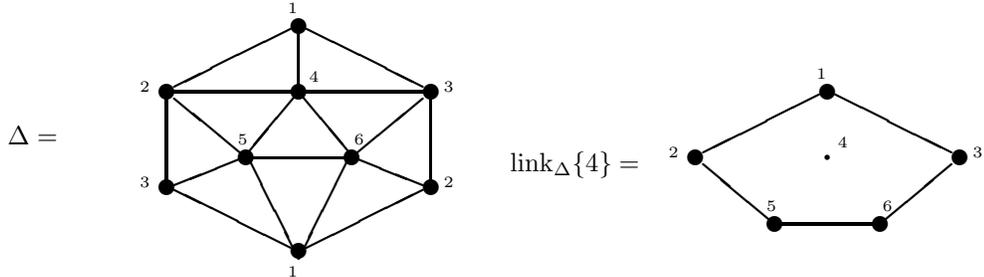
\vspace{2mm}
\begin{center}
\begin{picture}(400,80)
  \put(40, 30){$\Delta=$}
  \thicklines
  \put(150,75){\circle*{6}}  %1
  \put(100,50){\circle*{6}}  %2
  \put(100,14){\circle*{6}}  %3
  \put(150,-10){\circle*{6}}  %4
  \put(200,14){\circle*{6}}  %5
  \put(200,50){\circle*{6}}  %6
  \put(150,50){\circle*{6}}  %7
  \put(130,25){\circle*{6}}  %8
  \put(170,25){\circle*{6}}  %9
  \put(146,80){{\tiny $1$}}  %1
  \put(90,50){{\tiny $2$}}  %2
  \put(90,14){{\tiny $3$}}  %3
  \put(146,-20){{\tiny $1$}}  %4
  \put(205,14){{\tiny $2$}}  %5
  \put(205,50){{\tiny $3$}}  %6
  \put(154,54){{\tiny $4$}}  %7
  \put(127,30){{\tiny $5$}}  %8
  \put(171,30){{\tiny $6$}}  %9
  \put(147,74){\line(-2,-1){44}}  %{1, 2}
  \put(153,74){\line(2,-1){44}}  %{1, 6}
  \put(147,-9){\line(-2,1){44}}  %{4, 3}
  \put(153,-9){\line(2,1){44}}  %{4, 5}
  \put(100,50){\line(0,-1){36}}  %{2, 3}
  \put(200,50){\line(0,-1){36}}  %{6, 5}  
  \put(150,50){\line(0,1){25}}  %{7, 1}
  \put(147,50){\line(-1,0){47}}  %{7, 2}
  \put(153,50){\line(1,0){47}}  %{7, 6}
  \put(150,50){\line(-5,-6){20}}  %{7, 8}
  \put(150,50){\line(5,-6){20}}  %{7, 9}
  \put(130,25){\line(-6,5){27}}  %{8, 2}
  \put(129,25){\line(-5,-2){26}}  %{8, 3}
  \put(131,25){\line(1,-2){19}}  %{8, 4}
  \put(130,25){\line(1,0){40}}  %{8, 9}
  \put(169,25){\line(-1,-2){19}}  %{9, 4}
  \put(171,25){\line(5,-2){26}}  %{9, 5}
  \put(170,25){\line(6,5){27}}  %{9, 6}
%%%%%%
  \put(230,20){$\link_{\Delta}\{4\}=$}
    \thicklines
  \put(350,50){\circle*{6}}  %1
  \put(300,25){\circle*{6}}  %2
  \put(400,25){\circle*{6}}  %6
  \put(350,25){\circle*{2}}  %7
  \put(330,0){\circle*{6}}  %8
  \put(370,0){\circle*{6}}  %9
  \put(346,55){{\tiny $1$}}  %1
  \put(290,25){{\tiny $2$}}  %2
  \put(405,25){{\tiny $3$}}  %6
  \put(354,29){{\tiny $4$}}  %7
  \put(327,5){{\tiny $5$}}  %8
  \put(371,5){{\tiny $6$}}  %9
  \put(347,49){\line(-2,-1){44}}  %{1, 2}
  \put(353,49){\line(2,-1){44}}  %{1, 6}  
  \put(330,0){\line(-6,5){27}}  %{8, 2}
  \put(330,0){\line(1,0){40}}  %{8, 9}
  \put(370,0){\line(6,5){27}}  %{9, 6}
\end{picture}
\end{center}
\par \vspace{6mm}
Since 
$\widetilde{\chi}(\Delta) = -1 +f_0 - f_1 + f_2 = -1 + 6 - 15 +10 =0 \ne (-1)^2$,   
$K[\Delta]$ is \textit{not} Gorenstein for any field $K$. 
Moreover, Reisner proved that $K[\Delta]$ is  Cohen-Macaulay if and only if 
$\chara K \ne 2$. 
\par  
The link of every vertex is a pentagon, and $\Delta^{(1)}$ is the complete $6$-graph. 
Hence it follows from Corollary \ref{S2thm} that $S/I_{\Delta}^{(2)}$ 
has $(S_2)$.  
But it is \textit{not} Cohen-Macaulay; see  \cite[Example 2.8]{MiT2}.  
\par 
One can easily see that $x_1x_2x_3x_4x_5x_6 \in 
I_{\Delta}^{(2)} \setminus I_{\Delta}^2$.   
Hence $S/I_{\Delta}^2$ does not satisfy $(S_2)$. 
\end{exam}

\begin{quest} \label{S2-edge}
Let $I(G)$ be the edge ideal of a graph $G$. 
If $S/I(G)^{(2)}$ satisfies $(S_2)$,  
then is it Cohen-Macaulay?  
\end{quest}

%%%%%%%%%%%%%%%%%%%%%%%%%%%%%%%%%%%%%%%%%%%%%%%%%%%%%%%%%%%%%%%
%%%%%%%%%%%%%%%%%%%%%%%
\medskip
\section{When does $I^{(2)} = I^2$ hold}
\par 
In this section, we discuss when $I^{(2)} = I^2$ holds
for any squarefree monomial ideal $I$.
First we introduce the notion of special triangles.  

\begin{defn} \label{Hypergraph} 
Let $I$ be a squarefree monomial ideal of $S=K[x_1,\ldots,x_n]$.  
Let $G(I) = \{x^{H_1},\ldots,x^{H_\mu}\}$ be the minimal set of 
monomial generators, where $x^{H} = x_{i_1}\cdots x_{i_r}$ 
for $H=\{i_1,\ldots,i_r\}$.   
Then $\mathcal{H}(I)$ is called the \textit{associated hypergraph} of $I$ 
if the vertex set of $\mathcal{H}(I)$ is $V$ and 
the edge set is $\{H_1,\ldots,H_\mu\}$.   
\par
Then $\{i,j,k\}$ is called a \textit{special triangle} of $\mathcal{H}(I)$ if 
there exist $H_i,H_j,H_k \in \mathcal{H}(I)$ such that 
\[
H_i \cap \{i,j,k\} = \{j,k\},\qquad 
H_j \cap \{i,j,k\} = \{i,k\},\qquad 
H_k \cap \{i,j,k\} = \{i,j\}. 
\]
Then we say that \lq\lq $H_i,H_j,H_k$ \textit{make a special triangle}
 $\{i,j,k\}$''. 
\end{defn}

\par 
For instance, if $G(I)$ contains $x_1x_2L_1$, $x_2x_3L_2$, $x_3x_1L_3$ 
($L_1,L_2,L_3$ are monomials any of which is not divided by $x_1$,$x_2$ nor $x_3$),
then $\{1,2,3\}$ is a special triangle. 

\begin{remark}
A special cycle is considered in \cite{HHTZ}, 
and they prove that $I^{(\ell)}=I^{\ell}$ hold for any $\ell \ge 1$ if
there exists no special odd cycle in $\mathcal{H}(I)$. 
\end{remark}

\par \vspace{2mm}
The following is the main theorem in this section.  

\begin{thm} \label{SpTriangle}
Let $I$ be a squarefree monomial ideal. 
Then the following conditions are equivalent$:$
\begin{enumerate}
\item $I^{(2)} = I^2$ holds.
\item If there exist $\{H_1,H_2,H_3\} \subseteq \mathcal{H}(I)$ such that 
$H_1,H_2,H_3$ make a special triangle, 
then $x^{H_1 \cap H_2 \cap H_3} x^{H_1 \cup H_2 \cup H_3} \in I^2$. 
\end{enumerate}
\end{thm} 

\begin{remark}
If there exist no special triangles, then we have $I^{(2)}=I^2$. 
The converse is not true. 
\end{remark}

\par 
The following criterion is well known; see \cite{RTY}.  

\begin{cor} \label{Triangle}
Let $I(G)$ denote the edge ideal of a graph $G$. 
Then $I(G)^{(2)} = I(G)^2$ holds if and only if 
$G$ has no triangles $($the cycles of length $3$$)$. 
\end{cor}

\par
In what follows, we prove the above theorem.  
First we prove the following lemma. 

\begin{lemma} \label{SpMono}
Suppose that the condition $(2)$ in Theorem $\ref{SpTriangle}$ holds. 
Then $xI \cap (I^2 \colon x) \subseteq I^2$ holds for every $x \in V$.  
\end{lemma}

\begin{proof}
Suppose that there exist a variable $x_1$ and a monomial $M$ such that 
$M \in x_1 I \cap (I^2 \colon x) \setminus I^2$. 
As $x_1M \in I^2$, we can take $N_2$, $N_3 \in G(I)$ and a monomial $L$ such that 
\begin{eqnarray} \label{eq1}
x_1M = N_2N_3L. 
\end{eqnarray} 
On the other hand, as $M \in x_1I$, we can choose $N_1 \in G(I)$ 
and a monomial $L'$ such that 
\begin{eqnarray} \label{eq2}
M=N_1L' \quad \text{and} \quad  x_1 \,|\,L'. 
\end{eqnarray}
\begin{description}
\item[Claim 1] $x_1\,|\,N_2$, $x_1\,|\,N_3$ but $x_1 \,\not |\, N_1$. 
\end{description}
\par
As $M \notin I^2$, $x_1$ does not divide $L$. 
By Eqs.(\ref{eq1}),(\ref{eq2}), $N_2N_3L$ is divided by $x_1^2$. 
Hence $x_1$ divides both $N_2$ and $N_3$ because $N_i$ 
is a squarefree monomial for $i=2,3$.  
By a similar reason, we have that $N_1$ is not divided by $x_1$. 

\begin{description}
\item[Claim 2] $N_2\ne N_3$ and $\gcd(N_2,N_3) \,|\, L'$.  
\end{description}
\par
If $N_2=N_3$, then $x_1N_1L'=N_3^2L$ is divided by $x_1N_1$ 
and thus $N_3L$ is divided by $x_1N_1$. 
Then $M=N_1N_2(N_3L/x_1N_1) \in I^2$. This is a contradiction. 
Hence $N_2 \ne N_3$. 
\par
Since $x_1N_1L'=N_2N_3L$ is divided by $\gcd(N_2,N_3)^2$, $L'$ is  
divided by $\gcd(N_2,N_3)$ because $x_1N_1$ is squarefree. 

\begin{description}
\item[Claim 3] There exist variables $x_2$, $x_3$  such that
\[
x_2 \,\big|\, \frac{N_3}{\gcd(N_2,N_3)},\quad
x_3 \,\big|\, \frac{N_2}{\gcd(N_2,N_3)}, \quad x_2,x_3 \,|\,N_1
\]
\end{description}
\par
Note that any variable which divides $N_i$ 
for $i=2,3$ is a factor of $N_1$ or $L'$. 
Since $L' \notin I$, $L'/\gcd(N_2,N_3)$ is not divided by $N_3/\gcd(N_2,N_3)$. 
Thus there exists a variable $x_2$ such that 
$x_2\,|\, N_3 /\gcd(N_2,N_3)$ and $x_2\,|\,N_1$. 
The other statement follows from a similar argument.    
\par \vspace{2mm}
Take $H_i \in \mathcal{H}(I)$ such that $x^{H_i} = N_i$ for each $i=1,2,3$. 

\begin{description}
\item[Claim 4] $H_1,H_2,H_3$ make a special triangle $\{1,2,3\}$.   
\end{description}
\par
The assertion immediately follows from Claim 1 and Claim 3. 
By the Claim 4, we get a contradiction. 

\par\vspace{2mm}
By assumption, we get 
\[
 \gcd(N_1,N_2,N_3)\sqrt{N_1N_2N_3}
= x^{H_1 \cap H_2 \cap H_3} \cdot x^{H_1 \cup H_2 \cup H_3} \in I^2,
\]
where $\sqrt{N}=x_{i_1}\cdots x_{i_r}$ for a monomial 
$N=x_{i_1}^{a_{i_1}}\cdots x_{i_r}^{a_{i_r}}$ $(a_{i_j}>0)$. 
Since $N_1$ divides $N_2N_3L$ and $x_1 \,|\,N_2,\;N_3$, 
we have 
\begin{equation} \label{eq:sqrt}
\sqrt{N_1N_2N_3}\;\big|\; \frac{N_2N_3L}{x_1}=M.
\end{equation} 
On the other hand, since $x_1 \not | \gcd(N_1,N_2,N_3)$, we have 
\begin{equation} \label{eq:gcd}
\gcd(N_1,N_2,N_3)^2 \;\big|\;\frac{N_2N_3}{x_1} \;\big|\;M.  
\end{equation}
Hence Eqs. (\ref{eq:sqrt}), (\ref{eq:gcd}) imply 
\[
\gcd(N_1,N_2,N_3)\sqrt{N_1N_2N_3}\;\big|\; M. 
\]  
Therefore $M \in I^2$, which contradicts the choice of $M$.  
\end{proof}

\par
Now suppose that $I^{(2)}_{x}=I^2_{x}$ holds for every vertex $x \in V$.
Then $I^{(2)} = I^2$ if and only if $\frm \notin \Ass(S/I^2)$. 
Hence the following lemma is useful when we use an induction. 

\begin{lemma}[{\rm See the proof of \cite[Theorem 5.9]{SVV}}] \label{SVVproof}
Let $I$ be a squarefree monomial ideal of $S$ with $\dim S/I \ge 1$. 
Now suppose that $xI \cap (I^2 \colon x)\subseteq I^2$ for every variable $x$. 
Then $\frm \notin \Ass_S(S/I^2)$. 
\end{lemma} 

\begin{proof}
Since $I^2$ and $\frm$ are monomial ideals, it suffices to show 
$I^2 \colon M \ne \frm$  for every variable $x$ and any monomial $M$.  
\par
Now suppose that $I^2 \colon M = \frm$   
for some monomial $M \notin I^2$. 
Since $\frm M \subseteq I^2 \subseteq I$ and $\depth S/I > 0$, we have 
$M \in I$. 
So we may assume that $M=x_1\cdots x_k L$, 
where $N=x_1 \cdots x_k \in G(I)$ and $L$ is a monomial. 
By assumption, $x_kM = x_1(x_2 \cdots x_{k-1}x_k^2 L) \in I^2$. 
Since $I$ is generated by squarefree monomials, we then have 
$x_2\cdots x_{k-1}x_k^2L \in I$ and hence 
$x_2\cdots x_{k-1}x_kL \in I$. 
Hence $M \in x_1I \cap (I^2 \colon x_1) \subseteq I^2$. 
This is a contradiction. 
\end{proof}

\begin{proof}[Proof of Theorem $\ref{SpTriangle}$]
First we show $(2) \Longrightarrow (1)$. 
Suppose (2). 
Since this condition preserves under localization, we may assume that 
$(I^{(2)})_{x}=(I^{2})_{x}$ for any variable $x$ by an induction on $\dim S/I$. 
By the above two lemmata, we have $\frm \notin \Ass_S(S/I^2)$.
Hence $I^{(2)}=I^2$, as required.  
\par \vspace{2mm}
Next we show $(1) \Longrightarrow (2)$. 
Suppose that 
there exists a subset $\{H_1,H_2,H_3\} \subseteq \mathcal{H}(I)$ such that 
$H_1,H_2,H_3$ make a special triangle and 
$x^{H_1\cap H_2 \cap H_3} x^{H_1 \cup H_2 \cup H_3} \notin I^2$. 
Then it suffices to show $I^2 \subsetneq I^{(2)}$. 
\par
Put $H= H_1 \cup H_2 \cup H_3$. 
Let $I_H$ be the squarefree monomial ideal of 
$K[x\,:\,x \in V \setminus H]$ such that 
$I_H S + (x \in V \setminus H)=I+(x \in V \setminus H)$.
Let $P$ be any minimal prime ideal of $I_H$. 
If $\height P =1$, then there exists a vertex $j \in H_1 \cap H_2 \cap H_3$ such that 
$P=(x_j)$. Then $M:=x^{H_1 \cap H_2 \cap H_3} x^{H} \in (x_j^2)=P^2$. 
If $\height P\ge 2$, then $P$ contains two variables $x_i,x_j$ with $i,j \in H$. 
Then $x^H \in P^2$ and hence $M \in P^2$. 
Therefore $M \in I_H^{(2)}$ but $M \notin I_H^2$ by the assumption that 
$M \notin I^2$. 
\end{proof}

\medskip
Suppose $U \cap V = \emptyset$. 
Let $\Gamma$ (resp. $\Lambda$) be a simplicial complex on $U$ (resp. $V$). 
Then the \textit{simplicial join} of $\Gamma$ and $\Lambda$, 
denoted by $\Gamma * \Lambda$, is defined by  
$\Gamma * \Lambda = \{F \cup G \,:\, F \in \Delta,\; G \in \Lambda\}$.  
It is a simplicial complex on $U \cup V$. 
\par 
The following corollary 
is probably well-known (and hence so is Corollary \ref{Disjoint}), 
but we give a proof as an application of Theorem \ref{SpTriangle}.

\begin{cor} \label{Joincor}
Let $\Gamma$ be a simplicial complex on $U$ 
and $\Lambda$ a simplicial complex on $V$. 
Let $\Delta = \Gamma * \Lambda$ denote 
the simplicial join of $\Gamma$ and $\Lambda$. 
Then $\Delta$ is a simplicial complex on $W =U \coprod V$. 
Put $R=K[U]$, $S=K[V]$ and $T=R \otimes_K S \cong K[W]$. 
Then$:$
\begin{enumerate}
\item $I_{\Delta}^{(2)}=I_{\Delta}^2$ if and only if 
$I_{\Gamma}^{(2)}=I_{\Gamma}^2$ and $I_{\Lambda}^{(2)}=I_{\Lambda}^2$. 
\item $T/I_{\Delta}^{2}$ is Cohen--Macaulay if and only if   
so do $R/I_{\Gamma}^{2}$ and $S/I_{\Lambda}^{2}$.  
\end{enumerate}
\end{cor}

\begin{proof}
(1) Note that $I_{\Delta} = I_{\Gamma}T + I_{\Lambda}T$ and 
$G(I_{\Delta})$ is a disjoint union of $G(I_{\Gamma})$ and $G(I_{\Lambda})$. 
Thus it immediately follows from Theorem \ref{SpTriangle}.
\par
(2) It immediately follows from (1) and \cite[Theorem 2.7]{MiT2}. 
\end{proof}

\par \vspace{2mm}
A disjoint union of two graphs $G_1$ and $G_2$, denoted by 
$G_1 \coprod G_2$, is the graph $G$ which satisfies 
$V(G) = V(G_1) \cup V(G_2)$ and $E(G)=E(G_1) \cup E(G_2)$.   
Let $G = G_1 \coprod \ldots \coprod G_r$ be a disjoint union of 
graphs $G_1,\ldots,G_r$, and 
let $\Delta_i$ (resp. $\Delta$) be the complementary simplicial 
complex of $G_i$ 
for each $i=1,\ldots,r$ (resp. $G$).   
Then $\Delta$ is equal to the simplicial join 
$\Delta_1 * \cdots * \Delta_r$. 

\begin{cor} \label{Disjoint}
Let $G=G_1 \coprod \ldots \coprod G_r$ be a disjoint union of 
graphs $G_i$ for which  $I(G_i)^2$ is a Cohen-Macaulay ideal. 
Then $I(G)^2$ is a Cohen--Macaulay ideal. 
\end{cor}

\begin{exam} \label{DisjointPentagon}
Let $G=G_1 \coprod \ldots \coprod G_r$ be a disjoint union of 
the pentagons $G_i$ for $i=1,\ldots,r$. 
Then $I(G)^2$ is a Cohen--Macaulay ideal. 
\end{exam}

\begin{proof}[\quad Proof]
It follows that the second symbolic power of the edge ideal of the pentagon 
is a Cohen--Macaulay ideal. 
\end{proof}

%%%%%%%%%%%%%%%%%%%%%%%%%%%%%%%%%%%%%%%%%%%%%%%%%%%%%%%%%%%%%%%%%%%%%%%%%%%%%%%%%%%%%
\medskip %\newpage
\section{Examples of Stanley-Reisner ideals whose square is Cohen-Macaulay}

\par
By Corollary \ref{Joincor} we know that there exists a 
simplicial complex $\Delta $ with arbitrary high dimension
such that $I_{\Delta}^2$  is non-trivially Cohen-Macaulay.
We now consider the following question. 

\begin{quest}
For a given integer $d\ge 2$,
is there a simplicial complex $\Delta$ with $\dim \Delta=d-1$ 
such that $S/I_{\Delta}^2$ is Cohen-Macaulay and such that 
$\Delta$ cannot be expressed as the simplicial join of two non-empty 
complexes?
\end{quest}

\par 
We give two families of examples as affirmative answers, using liaison theory.
The following key proposition is due to 
Buchweitz \cite{Bu}; see also  Kustin and Miller \cite{KM2}. 
Note that it gives a partial converse of 
Theorem \ref{Buchsbaum}.   

\begin{prop}[\textrm{cf. \cite[6.2.11]{Bu}, \cite[Proposition 7.1]{KM2}}] 
\label{Kustin}
Let $I$ be a Gorenstein homogeneous ideal in a polynomial ring $S$. 
Assume that there exist a homogeneous polynomial ring 
$T=S[z_1,\ldots,z_r]$ $(\deg z_i =1)$ and a homogeneous radical ideal $L$ such that  
\begin{enumerate}
\item[(a)] $S/I \cong T/(z_1,\ldots,z_r, L)$. 
\item[(b)] $z_1,\ldots,z_r$ is a regular sequence on $T/L$.  
\item[(c)] $L$ is in the linkage class of a complete intersection in $T$.   
\end{enumerate}
Then $S/I^2$ is Cohen-Macaulay. 
\end{prop}

\begin{proof} 
Since $S/I^2$ is isomorphic to the ring $T/(z_1,\ldots,z_r,L^2)$, it is enough to show 
that  $T/L^2$ is Cohen-Macaulay.  
\par 
Let $\fraM$ be the unique homogeneous maximal ideal of $T$, and set 
$R=\widehat{T_{\fraM}}$, the $\fraM$-adic completion of $T_{\fraM}$.  
As $R/LR$ is a radical Gorenstein ideal, we can conclude that 
$LR/(LR)^2$ is Cohen-Macaulay, and thus $R/(LR)^2$ is Cohen-Macaulay 
by \cite[Proposition 7.1]{KM2}. 
It follows from Matijevic-Roberts theorem that $T/L^2$ is Cohen-Macaulay, as required. 
\end{proof}

\par
It is well-known 
that any Gorenstein ideal of codimension $3$ lies in the linkage 
class of a complete intersection; see \cite{BE, Wat} or 
\cite[Theorem 4.15]{Vas}.   
Thus we can obtain the following corollary.  

\begin{cor} \label{codim3}
Let $I_{\Delta} \subseteq S$ be a Gorenstein Stanley-Reisner ideal of 
codimension $3$. Then $S/I_{\Delta}^2$ is Cohen-Macaulay. 
\end{cor}

\medskip
In the rest of this section 
we prove the second power of the Stanley-Reisner ideal 
of a stellar subdivision of any non-acyclic complete intersection complex  
is Cohen-Macaulay. 
In what follows, 
as vertices of simplicial complexes we use indeterminates 
instead of natural numbers for convenience.
Let $\Gamma$  be a non-acyclic complete intersection simplicial complex 
whose Stanley-Reisner ideal is 
\[
I_{\Gamma}=(x_{11}x_{12}\cdots x_{1i_{1}}, 
x_{21}x_{22}\cdots x_{2i_{2}}, \ldots,
x_{\mu 1}x_{\mu 2}\cdots x_{\mu i_{\mu }}).
\]
\par
Let $\mathcal{F}(\Gamma)$ be the set of all facets of $\Gamma$.
Then
\begin{eqnarray*}
\mathcal{F}(\Gamma)=
& \{ & \{x_{11}, \ldots , \widehat{x_{1k_{1}}}, \ldots, x_{1i_{1}}, 
x_{21}, \ldots, \widehat{x_{2k_{2}}}, \ldots,  x_{2i_{2}}, \ldots,\\
&&x_{\mu 1}, \ldots , \widehat{x_{\mu k_{\mu }}}, \ldots,  
x_{\mu i_{\mu }} \}\\
&&\mid
1 \le k_1 \le i_1, 1 \le k_2 \le i_2, \dots, 1 \le k_{\mu } \le i_{\mu }
\}.
\end{eqnarray*}

\par
Let $\Delta$  be the \textit{stellar subdivision} of $\Gamma$ on 
\[
F=\{ x_{11}, \ldots ,x_{1j_{1}}, x_{21}, \ldots ,x_{2j_{2}}, \ldots,
x_{p 1}, \ldots , x_{p j_{p }} \},
\]
where $1 \le p \le \mu$ and $1 \le j_1 < i_1, \dots , 1 \le j_p < i_p$ and $ j_1 +\cdots +j_p \ge 2$.
\par
Let $v$ be the new added vertex.
Then
\begin{eqnarray*}
\mathcal{F}(\Delta)= &\{ & G \in \mathcal{F}(\Gamma  ) \mid G \not\supset F \ \ \}
\cup \{ \{v \} \cup G \setminus \{w\} \mid G \supset F , w \in F \}\\
=& \{ & \{x_{11}, \ldots , \widehat{x_{1k_{1}}}, \dots, x_{1i_{1}}, 
x_{21}, \ldots, \widehat{x_{2k_{2}}}, \dots,  x_{2i_{2}}, \dots,\\
&&x_{\mu 1}, \ldots , \widehat{x_{\mu k_{\mu }}}, \dots,  
x_{\mu i_{\mu }} \}\\
&&\mid
1 \le k_1 \le i_1, 1 \le k_2 \le i_2, \dots, 1 \le k_{\mu } \le i_{\mu }\\
&& \mbox{ with } 1 \le k_1 \le j_1 \mbox{ or }  1 \le k_2 \le j_2 \mbox{ or }   
\dots \mbox{ or }  1 \le k_{p } \le j_{p}
\}\\
 \cup & \{ &
\{v, x_{11}, \dots , \widehat{x_{1k_{1}}}, \dots, x_{1i_{1}}, 
x_{21}, \dots, \widehat{x_{2k_{2}}}, \dots,  x_{2i_{2}}, \dots,\\
&&x_{\mu 1}, \dots , \widehat{x_{\mu k_{\mu }}}, \dots,  
x_{\mu i_{\mu }} \}
\setminus \{w\} \\
&&\mid
j_1+1  \le k_1 \le i_1, j_2+1  \le k_2 \le i_2, \dots, j_p+1  \le k_p \le i_p\\
&&1 \le k_{p+1 } \le i_{p+1}, \dots, 
1 \le k_{\mu } \le i_{\mu }, w \in F 
\}
\end{eqnarray*}
\par \noindent 
and 
\[
I_{\Delta}=( I_{\Gamma}, x_F, vx_{1j_1+1}\cdots 
x_{1i_{1}}, vx_{2j_2+1}\cdots x_{2i_{2}}, \dots ,
vx_{pj_p+1}\cdots x_{pi_{p}})
\]
is an ideal of a polynomial ring
\[
S=k[x_{11},\ldots, x_{1i_{1}}, 
x_{21},\ldots, x_{2i_{2}}, \ldots,
x_{\mu 1},\ldots, x_{\mu i_{\mu }},v]. 
\]

\par 
Applying Proposition \ref{Kustin} to this ideal $I=I_{\Delta}$, we obtain 
the following theorem. 
It is proved the two-dimensional case in \cite{TrTu}.

\begin{thm}  \label{Subdiv}
Let $\Delta=\Gamma_F$  be the stellar subdivision of the non-acyclic 
complete intersection complex $\Gamma$ as above. 
Then $S/I_{\Delta}^2$ is Cohen--Macaulay. 
\end{thm}

\begin{proof}
Consider the variables $\underline{z} = z_1,z_2,\ldots,z_N$, where $N = j_1+\cdots +j_p-1$
and put $Z= z_1\cdots z_N$.  
Moreover, we set  
\[
\begin{array}{rclcrcl}
X_1 & = & x_{1,1}\cdots x_{1,j_1},& & Y_1 & = & x_{1,j_1+1}\cdots x_{1,i_1}, \\ 
X_2 & = & x_{2,1} \cdots x_{2,j_2},& &  Y_2 & = & x_{2,j_2+1}\cdots x_{1,i_2}, \\
&\vdots & & &  & \vdots & \\ 
X_p & = & x_{p,1}\cdots x_{p,j_p} & & Y_p & = & x_{p,j_p+1}\cdots x_{p,i_p}, \\
& & & & Y_{p+1} & = & x_{p+1,1}\cdots x_{p+1,i_{p+1}}, \\
& & & & & \vdots & \\
& & & & Y_{\mu} & = & x_{\mu,1}\cdots x_{\mu,i_{\mu}}. \\
\end{array}
\]
and 
\[
 L = (I_{\Gamma},vY_1,\ldots,vY_p,vZ-x_F)
\subseteq T=S[\underline{z}].
\] 
Then $I_{\Gamma} = (X_1Y_1,\ldots,X_pY_p,Y_{p+1},\ldots,Y_{\mu})$, 
$I_{\Delta} = (I_{\Gamma}, x_F,vY_1,\ldots,vY_p)$ and 
$S/I_{\Delta}$ is isomorphic to $T/(\underline{z},L)$. 
\par
In what follows, we show that $L$ lies in the linkage class of a complete 
intersection (i.e., licci). 
Firstly, we can easily prove the following equality:
\begin{equation} \label{First-link}
(I_{\Gamma},Z) \colon (Y_1,\ldots,Y_{\mu},Z) = (I_{\Gamma},Z,x_F). 
\end{equation}
Secondly we show the following equality:
\begin{equation} \label{Second-link}
L=(I_{\Gamma},vZ-x_F) \colon (I_{\Gamma},Z,x_F). 
\end{equation}
To end this, it is enough to show the right-hand side is contained in $L$. 
Let $\alpha \in (I_{\Gamma},vZ-x_F) \colon (I_{\Gamma},Z,x_F)$. 
Then there exists a $\beta \in T$ such that $\alpha Z - \beta (vZ-x_F) \in I_{\Gamma}$. 
Then $\beta \in (I_{\Gamma},Z) \colon x_F = (Y_1,\ldots,Y_{\mu},Z)$. 
In particular, we can write $\beta = \sum_{i=1}^{\mu} \gamma_i Y_i + \delta Z$ for 
some $\gamma_i$, $\delta \in T$. 
It follows that 
\[
Z\bigg[\alpha - \sum_{i=1}^p \gamma_i (vY_i) - \delta (vZ-x_F)\bigg] 
\in I_{\Gamma}. 
\]  
As $Z$ is a nonzero divisor on $T/I_{\Gamma}T$, we conclude that $\alpha \in L$.  
\par 
In Equations (\ref{First-link}), (\ref{Second-link}), 
both $(I_{\Gamma},Z)$ and $(I_{\Gamma},vZ-x_F)$ 
are complete intersection ideals of the same height $\mu+1$ 
as $(Y_1,\ldots,Y_{\mu},Z)$ or $L$. 
Hence $L$ is licci. %lies in the linkage class of a complete intersection. 
\par
In order to prove that $S/I_{\Delta}^2$ is Cohen-Macaulay by 
Proposition \ref{Kustin}, 
it is enough to show that $\underline{z}$ is a regular sequence on $T/L$ 
and that $T/L$ is reduced. 
By the above proof, we have that $L$ is licci and $\dim T/L = \dim T/(Y_1,\ldots,Y_{\mu},Z)$. 
In particular, $L$ is Cohen-Macaulay and $\dim T/L = i_1+\cdots + i_{\mu} - \mu +N$. 
\par
On the other hand, 
\[
 \dim T/(\underline{z},L) = \dim S/I_{\Delta}
 = \dim S/(I_{\Gamma},v) = i_1 + \cdots + i_{\mu} - \mu 
= \dim T/L - N.
\]
This implies that $\underline{z}$ is a regular sequence on $T/L$. 
Moreover, as $T/(\underline{z},L)$ is reduced, so is $T/L$, as required. 
\end{proof}

\begin{remark}
The above Gorenstein ideals are obtained from the 
so-called Herzog ideals (see \cite{He, Hu, KM1, KM2}) 
and $T/L$ is called the \textit{Kustin-Miller unprojection ring} (\cite{BP}). 
Moreover, the assertion of Theorem \ref{Subdiv} 
says that the quotient algebras 
of those ideals are \textit{strongly unobstructed}.   
\end{remark}

\begin{exam}[Cross Polytope] \label{cross}
Let ${\bf e}_1,\ldots,{\bf e}_d$ be the fundamental vectors of 
the $d$-dimensional Euclidean space $\mathbb{R}^d$. 
Then the convex hull $\mathcal{P} = 
\mathrm{CONV}(\{\pm{\bf e}_1,\pm{\bf e}_2,\ldots,\pm{\bf e}_d \})$ 
is called the \textit{cross $d$-polytope}.
Let $\Gamma$ be the boundary complex of the cross $d$-polytope $\mathcal{P}$. 
Let $W=\{x_1,\ldots,x_d,y_1,\ldots,y_d\}$. 
For a sequence ${\bf i} = [i_1,\ldots,i_m]$ with $1 \le i_1 < \cdots < i_m \le d$, we assign  
a subset of $W$
\[
 F_{\bf i} = \big\{x_{i_1},\ldots,x_{i_m} \big\} \cup 
\big\{y_j \,:\, j \in [d] \setminus \{i_1,\ldots,i_m\}\big\}.   
\]
Then $\Gamma$ can be regarded as a simplicial complex on $W$ such that 
\[
\mathcal{F}(\Gamma) 
= \{F_{\bf i}: m=0,1,\ldots,d,\, 1 \le i_1 < \cdots < i_m \le d \},  
\] 
and it is a $(d-1)$-dimensional complete intersection complex
with
\[
 I_{\Gamma} = (x_1y_1,x_2y_2,\ldots,x_dy_d). 
\] 
\par 
Let $v$ be a new vertex, and choose a facet 
$F_{[1,2.\ldots,d]}=\{x_1,\ldots,x_d\}$ of $\Gamma$. 
Let $\Delta$ be the stellar subdivision of $\Gamma$ on $F$. 
Then  $\Delta$ is a $(d-1)$-dimensional Gorenstein complex on $V=W \cup\{v\}$
and its geometric realization of $\Delta$ is homeomorphic to 
$\mathbb{S}^{d-1}$. 
The above theorem says that 
the second power of 
\[
 I=(x_1y_1,x_2y_2,\ldots,x_dy_d,\,vy_1,\ldots,vy_d,\,x_1x_2\cdots x_d)
\]
is Cohen-Macaulay, but the third power is not if $d \ge 2$ 
because the third power of the Stanley-Reisner ideal $(x_1y_1,x_2y_2,vy_1,vy_2,x_1x_2)$
of a pentagon is not.  
\end{exam}

\par \vspace{2mm}
In the last of the paper, we give candidates of edge ideals $I(G)$ 
for which $S/I(G)^2$ is Cohen--Macaulay (but $S/I(G)^3$ is not by \cite{RTY}). 
For the case that $n=2$ it is mentioned in \cite[Theorem 3.7 (iv)]{TrTu}.

\begin{conj} \label{secondpowerCM}
Let $G$ be a graph on the vertex set $V=\{x_1,x_2,\ldots,x_{3n+2}\}$ with 
\[
I(G)  = 
\big(x_1x_2, \; \{x_{3k-1}x_{3k},x_{3k}x_{3k+1},x_{3k+1}x_{3k+2},
x_{3k+2}x_{3k-2}\}_{k=1,2,\ldots,n},\;
\{x_{3\ell-3}x_{3\ell}\}_{\ell=2,3,\ldots,n} \big).
\]
Then $S/I(G)^2$ is Cohen--Macaulay but $S/I(G)^3$ is not. 

\vspace{3mm}
\begin{center}
\end{center}
\end{conj}

%%%%%%%%%%%%%%%%%%%%%%%%%%%%%%%%%%%%%%%%%%%%%%%%%%%%%%%%%%%%%%%
%%%%%%%%%%%%%%%%%%%%%%%%%%%%%%%%%%%%%%%%%%%%%%%%%%%%%%%%%%%%%%%

\par \vspace{2mm}
\begin{acknowledgement}
We would like to thank the referee for his/her advice. 
Especially, Section 5 is arranged 
based on the referee's advice.  
Moreover, 
the second author was supported by JSPS 20540047. 
The third author was supported by JSPS 19340005. 
\end{acknowledgement}

%\newpage
%%%%%%      reference 
%%%%% %    

\end{document}